\documentclass{amsart}
\usepackage{biblatex}
\usepackage{hyperref}
\usepackage{mathtools}
\usepackage{amssymb}
\usepackage{breqn} 
\usepackage{xifthen}
\hypersetup{
    colorlinks,
    linkcolor={red},
    citecolor={green},
    urlcolor={blue}
}

\makeatletter
\@namedef{subjclassname@1991}{2020 Mathematics Subject Classification}
\makeatother

\newtheorem{prop}{Proposition}

\newtheorem{thm}{Theorem}

\newtheorem{cor}{Corollary}

\theoremstyle{definition}
\newtheorem{defn}{Definition}
\theoremstyle{remark} 
\newtheorem{remark}{Remark}
\newtheorem{example}{Example}

\newcommand{\abs}[1]{\left\lvert#1\right\rvert} 
\newcommand{\CC}{\mathbb{C}} 
\newcommand{\F}{\mathbb{F}} 
\newcommand{\N}{\mathbb{N}} 
\DeclareMathOperator{\form}{Form} 
\DeclareMathOperator{\im}{Im} 
\DeclareMathOperator{\row}{Row} 
\DeclareMathOperator{\spn}{Span} 
\DeclareMathOperator{\eval}{eval} 
\renewcommand{\hom}{\operatorname{Hom}} 
\renewcommand{\ker}{\operatorname{Ker}} 
\providecommand{\set}[2][]{
	\ifthenelse{\isempty{#1}}{
		\left\{#2\right\}
	}{
		\left\{\,#1\;\middle|\;#2\,\right\}}
	}

\begin{document}
\title{A multi-linear geometric estimate}
\author[C. Aten]{Charlotte Aten}
\address{Mathematics Department\\University of Rochester\\Rochester 14627\\USA}
\urladdr{\href{http://aten.cool}{http://aten.cool}}
\email{\href{mailto:charlotte.aten@rochester.edu}{charlotte.aten@rochester.edu}}
\author[A. Iosevich]{Alex Iosevich}
\address{Mathematics Department\\University of Rochester\\Rochester 14627\\USA}
\urladdr{\href{https://people.math.rochester.edu/faculty/iosevich/}{https://people.math.rochester.edu/faculty/iosevich/}}
\email{\href{alex.iosevich@rochester.edu}{alex.iosevich@rochester.edu}}
\subjclass{11E20}
\keywords{Multi-linear forms, sum-product conjecture, finite field Fourier analysis}

\begin{abstract}
We give a generalization of the geometric estimate used by Hart and the second author in their 2008 work on sums and products in finite fields. Their result concerned level sets of non-degenerate bilinear forms over finite fields, while in this work we prove that if \(E\subset\F_q^d\) is sufficiently large and \(\varpi\) is a non-degenerate multi-linear form then \(\varpi\) will attain all possible nonzero values as its arguments vary over \(E\), under a certain quantitative assumption on the extent to which \(E\) is projective. We show that our bound is nontrivial in the case that \(n=3\) and \(d=2\) and construct examples of sets to which this applies. In particular, we give conditions under which every member of \(\F_q^*\) belongs to \(A\cdot A\cdot A+A\cdot A\cdot A\cdot A\cdot A\cdot A\) where \(A\) is a union of cosets of a subgroup of \(\F_q^*\).
\end{abstract}

\maketitle

\tableofcontents

\section{Introduction}
One of the key problems of additive number theory is the celebrated sum-product conjecture due to Erd\H{o}s and Szemer\'{e}di. It says that if \(A\) is a finite subset of the integers of size \(n\), then either the sumset \(A+A\coloneqq\set[a+a']{a,a'\in A}\) or the product set \(A\cdot A\coloneqq\set[a\cdot a']{a,a'\in A}\) has size at least \(C_\epsilon n^{2-\epsilon}\) for any \(\epsilon>0\). Taking \(A=\set{1,2,\dots,n}\) or \(A=\set{2^1,2^2,\dots,2^n}\) shows that this estimate is, in general, best possible. In the former case, the sumset is small and the product set is large. In the latter case, it is the other way around. In spite of many decades of efforts by numerous mathematicians, the conjecture is far from being solved, the best known estimate, due to Konyagin and Shkredov\cite{KS2016}, being
    \[
        \max(\set{\abs{A+A},\abs{A\cdot A}})\ge Cn^{\frac{4}{3}+c}
    \]
for any \(c<\frac{5}{9813}\), improving upon a previous breakthrough due to Solymosi\cite{S2009}. 

This problem makes sense in any ring. For example, it has been extensively studied in the setting of finite fields. One of the key objects that arises in the study of this problem is the set
    \[
        A\cdot A+A\cdot A=\set[a\cdot a'+b\cdot b']{a,a',b,b'\in A}
    \]
When \(A\subset\F_q\), the finite field with \(q\) elements, it is known\cite{hart2008} that \(A\cdot A+A\cdot A=\F_q\) if \(\abs{A}>q^{\frac{3}{4}}\), and \(\abs{A\cdot A+A\cdot A}>\frac{q}{2}\) if \(\abs{A}>q^{\frac{2}{3}}\). It is widely conjectured that the latter conclusion should hold if \(\abs{A}\ge C_\epsilon q^{\frac{1}{2}+\epsilon}\) for any \(\epsilon>0\), but there has been no progress on this problem since 2008.

The exponents described in the previous paragraph were obtained by taking this one-dimensional problem into two dimensions by considering \(E=A\times A\) with \(A\subset\F_q\) and considering the following, higher dimensional version of the question. How large does \(E\subset\F_q^2\) need to be to ensure that
    \[
        \prod(E)\coloneqq\set[x\cdot y\coloneqq x_1y_1+x_2y_2]{x,y\in E}=\F_q,
    \]
or, at least, that \(\abs{\prod(E)}>\frac{q}{2}\)? The authors in \cite{hart2008} proved that the former conclusion holds if \(\abs{E}>q^{\frac{3}{4}}\) and the latter if \(\abs{E}>q^{\frac{4}{3}}\), with an additional assumption that holds, for example, if \(E=A\times A\). These results can be extended to general non-degenerate quadratic forms (see e.g. \cite{BHIPR2017}). 

The purpose of this paper is to generalize the estimates described above to multi-linear forms. This significantly expands the range of geometric applications and allows for interesting interactions of ideas from multi-linear algebra and combinatorial geometry. As a very simple example of the setup we are going to introduce, we will be able to effectively analyze expressions of the form
    \[
        A\cdot A\cdot A+A\cdot A\cdot A
    \]
along with considerably more complicated objects.

We begin by introducing some terminology pertaining to multi-linear forms.

\begin{defn}[Multi-linear form]
Given a vector space \(V\) over a field \(\F\) and some \(n\in\N\) we refer to a linear transformation
    \[
        \varpi\colon V^{\otimes n}\to\F
    \]
as a \emph{multi-linear form} (or just as a \emph{form}). More explcitly, we say that \(\varpi\) is an \emph{\(n\)-linear form on \(V\)}, that \(\varpi\) has \emph{arity} \(n\), or that \(\varpi\) is an \emph{\(n\)-form} on \(V\).
\end{defn}

\begin{remark}
Note the distinction between \(V^{\otimes n}\), the tensor product of \(n\) copies of \(V\), and \(V^n\), the direct product of \(n\) copies of \(V\). Recall that given a basis \(e_1,\dots,e_n\) of \(V\), the \(n\)-fold tensor product of \(V\) is generated by
    \[
        e_{i_1}\otimes e_{i_2}\otimes\cdots\otimes e_{i_n},
    \]
where \(e_{i_j}\) is one of the basis vectors of \(V\) given above. 

More precisely, let \(V\) and \(W\) be vector spaces over the same field. Form a vector space whose formal basis is \(V\times W\) (the Cartesian product of the sets \(V\) and \(W\)). We now form a quotient where we identify \((av,w)\) with \((v,aw)\), and with \(a(v,w)\). The resulting vector space is the tensor product \(V\otimes W\) of \(V\) and \(W\). This yields a map \(\boxtimes\colon V\times W\to V\otimes W\) which takes a pair \((v,w)\) to its equivalence class, which we denote by \(v\otimes w\). Similarly, given a set of vectors \(E\subset V\) we may then form \(E^{\boxtimes n}\), which is the image of \(E^n\) under the analogous map \(V^n\to V^{\otimes n}\). Note that when the subset \(E\) is actually a subspace then both the set \(E^{\boxtimes n}\) and the vector space \(E^{\otimes n}\) are defined. In this case we have that \(E^{\otimes n}\) is generated by \(E^{\boxtimes n}\) but in general \(E^{\otimes n}\) contains many more points than \(E^{\boxtimes n}\).
\end{remark} 

When doing calculations with an \(n\)-linear form \(\varpi\colon V^{\otimes n}\to\F\) we often write \(\varpi(x_1,\dots,x_n)\) rather than \(\varpi(x_1\otimes\cdots\otimes x_n)\). Note that a multi-linear form is an element of the dual space of \(V^{\otimes n}\). Multi-linear forms thus inherit the structure of a vector space.

\begin{defn}[Space of multi-linear forms]
Given a vector space \(V\) over a field \(\F\) and some \(n\in\N\) we denote by \(\form(V,n)\) the dual \(\F\)-vector space to \(V^{\otimes n}\). That is, \(\form(V,n)\coloneqq\hom(V^{\otimes n},\F)\).
\end{defn}

We reserve the notation \(V^*\) for the set of all nonzero vectors in \(V\), so we will not write \(V^*\) to indicate the dual space. In the case that \(V=\F_q^d\) for a prime power \(q\) we write \(\form(q,d,n)\coloneqq\form(\F_q^d,n)\). In this paper we primarily consider forms on finite-dimensional vector spaces over finite fields.

\begin{defn}[Level set]
Given an \(\F\)-vector space \(V\), a form \(\varpi\in\form(V,n)\), \(E\subset V\), \(t\in\F\) we define the \emph{\(t\)-level set} of \(\varpi\) (with respect to \(E\)) to be
    \[
        L_t\coloneqq\set[(z,w)\in E^{\boxtimes(n-1)}\times E]{\varpi(z,w)=t}
    \]
and we define \(\nu(t)\coloneqq\abs{L_t}\).
\end{defn}

The asymmetry between \(E^{\boxtimes(n-1)}\) and \(E\) in the preceding definition arises from our use of the following transformations.

\begin{defn}[Evaluation map]
Given a vector space \(V\), some \(n\in\N\), some \(k\in[n]\), and subspaces \(A\le V^{\otimes(n-1)}\) and \(B\le V\) the \emph{\(k^{\text{th}}\) evaluation map} on \((A,B)\) is
    \[
        \eval_{k,A,B}\colon\form(V,n)\otimes B\to\hom(A,\F)
    \]
is given by
    \[
        (\eval_{k,A,B}(\varpi\otimes y))(x_1,\dots,x_{n-1})\coloneqq\varpi(x_1,\dots,x_{k-1},y,x_{k+1},\dots,x_{n-1}).
    \]
\end{defn}

We simply write \(\eval_k\) rather than \(\eval_{k,A,B}\) in the case that \(A=V^{\otimes(n-1)}\) and \(B=V\). Observe that \(\eval_k\) is the linear transformation given by plugging a vector \(y\) into the \(k^{\text{th}}\) slot of a form \(\varpi\) and that a general \(\eval_{k,A,B}\) is obtained by restricting this map appropriately. We also define the linear transformation
    \[
        \eval_{k,A,B}^\varpi\colon B\to\hom(A,\F)
    \]
by \(\eval_{k,A,B}^\varpi(y)\coloneqq\eval_{k,A,B}(\varpi,y)\).

\begin{defn}[\((A,B)\)-non-degenerate form]
Given a form \(\varpi\in\form(V,n)\) and subspaces \(A\le V^{\otimes(n-1)}\) and \(B\le V\) we say that \(\varpi\) is \emph{\((A,B)\)-non-degenerate} in the \(k^{\text{th}}\) coordinate when \(\ker(\eval_{k,A,B}^\varpi)=0\).
\end{defn}

\begin{example}[Ternary dot product]
\label{example:ternary_dot_product}
Let \(V=\F_q^d\), \(A=V^{\otimes(n-1)}\), \(B=V\), \(n=3\), and
    \[
        \varpi(x,y,z)=x_1y_1z_1+x_2y_2z_2+\cdots+x_dy_dz_d.
    \]
It is not difficult to see that this form is \((A,B)\)-non-degenerate.

\vskip.125in 

If we keep \(B\) the same, change \(A\) to \(W^{\otimes(n-1)}\), where
    \[
        W=\set[x\in\F_q^d]{x_1=0},
    \]
and use the same form as above, we get an \((A,B)\)-degenerate form.
\end{example} 

In the special case that \(A=V^{\otimes(n-1)}\) and \(B=V\) the condition for \((A,B)\)-non-degeneracy is that \(\ker(\eval_k^\varpi)=0\). We refer to a form which is \((V^{\otimes(n-1)},V)\)-non-degenerate simply as a \emph{non-degenerate form}. Similarly, we say that a form \(\varpi\) is \emph{degenerate} when \(\varpi\) is not non-degenerate. A related notion of degeneracy for multi-linear forms is already present in the literature\cite[p.445]{gelfand}. 

Our results depend on the following parameter measuring the extent to which a set is closed under taking scalar multiples.

\begin{defn}[Projective index]
Given \(E\subset\F_q^d\) we say that \(E\) has \emph{projective index} \(\alpha\) when
    \[
        \frac{\abs{\set[(a,w)\in(\F_q^*\setminus{\set{1}})\times\F_q^d]{w,aw\in E}}}{(q-2)\abs{E}}\ge\alpha.
    \]
\end{defn}

Note that if \(E\) has projective index \(1\) then \(E\) is projective in the sense that \(E\cup\set{0}\) is a union of lines through the origin, while any set \(E\) has projective index \(0\).

We now have all the terminology needed to state our generalization of the aforementioned estimate of Hart and the second author.

\begin{thm}
\label{thm:generalized_result}
Suppose that \(\varpi\in\form(q,d,n)\) for some \(n\ge2\), that \(E\subset\F_q^d\), and that \(E\) has projective index \(\alpha\). If there exists an \(r\)-dimensional subspace \(A\) of \((\F_q^d)^{\otimes(n-1)}\) and a subspace \(B\) of \(\F_q^d\) such that
    \begin{enumerate}
        \item \(E^{\boxtimes(n-1)}\subset A\),
        \item \(E\subset B\)
        \item \(\varpi\) is \((A,B)\)-non-degenerate, and
        \item \(\abs{E}>q^{\frac{r+n-1}{n}}\left(1-\alpha\left(1-\frac{2}{q}\right)\right)^{\frac{1}{n}}\)
    \end{enumerate}
then \(\F_q^*\subset\varpi(E^n)\). This bound is sharp.
\end{thm}

The proof of the bound is located in \autoref{sec:proof_main_thm} and a sharp example is located in \autoref{sec:sharp_example}.

\section{Proof of the main theorem}
\label{sec:proof_main_thm}
Here we give the proof of \autoref{thm:generalized_result} except for the example showing that our result is sharp, which we produce in \autoref{sec:sharp_example}.

\begin{proof}
We argue the case where \(\varpi\) is \((A,B)\)-non-degenerate in the \(n^{\text{th}}\) coordinate, the other cases being equivalent by symmetry. Let \(\chi\colon\F_q\to\CC\) be a nontrivial additive character. We have that
    \[
        \nu(t)=\sum_{\mathclap{\substack{z\in E^{\boxtimes(n-1)}\\w\in E}}}q^{-1}\sum_{s\in\F_q}\chi(s(\varpi(z,w)-t)).
    \]
This implies that
    \[
        \nu(t)=q^{-1}\abs{E^{\boxtimes(n-1)}}\abs{E}+R
    \]
where
    \[
        R\coloneqq\sum_{\mathclap{\substack{z\in E^{\boxtimes(n-1)}\\w\in E}}}q^{-1}\sum_{s\in\F_q^*}\chi(s(\varpi(x)-t)).
    \]

View \(R\) as a sum in \(z\) and apply Cauchy-Schwarz. We obtain
	\begin{dmath*}
		R^2 \le \abs{E^{\boxtimes(n-1)}}\sum_{\mathclap{z\in E^{\boxtimes(n-1)}}}q^{-2}\sum_{s,s'\in\F_q^*}\sum_{w,w'\in E}\chi(s(\varpi(z,w)-t))\chi(s'(\varpi(z,w')-t))
		\le \abs{E^{\boxtimes(n-1)}}\sum_{z\in A}q^{-2}\sum_{s,s'\in\F_q^*}\sum_{w,w'\in E}\chi(s\varpi(z,w)-s'\varpi(z,w'))\chi(t(s'-s))
		= \abs{E^{\boxtimes(n-1)}}\sum_{z\in A}q^{-2}\sum_{s,s'\in\F_q^*}\sum_{w,w'\in E}\chi(\varpi(z,sw-s'w'))\chi(t(s'-s)).
	\end{dmath*}
Define
    \[
        U\coloneqq\abs{E^{\boxtimes(n-1)}}\sum_{z\in A}q^{-2}\sum_{\mathclap{\substack{s,s'\in\F_q^*\\w,w'\in E\\sw=s'w'}}}\chi(\varpi(z,sw-s'w'))\chi(t(s'-s))
    \]
and
    \[
        V\coloneqq\abs{E^{\boxtimes(n-1)}}\sum_{z\in A}q^{-2}\sum_{\mathclap{\substack{s,s'\in\F_q^*\\w,w'\in E\\sw\neq s'w'}}}\chi(\varpi(z,sw-s'w'))\chi(t(s'-s)).
    \]
We have that \(R^2\le U+V\).

In the case of \(U\) we have
	\begin{dmath*}
		U = \abs{E^{\boxtimes(n-1)}}q^{-2}\sum_{\mathclap{\substack{s,s'\in\F_q^*\\w,w'\in E\\sw=s'w'}}}\chi(t(s'-s))\sum_{z\in A}\chi(\varpi(z,0))
		= \abs{E^{\boxtimes(n-1)}}q^{r-2}\sum_{\mathclap{\substack{s,s'\in\F_q^*\\w,w'\in\F_q^d\\sw=s'w'}}}\chi(t(s'-s))E(w)E(w'),
	\end{dmath*}
while in the case of \(V\) we have
	\[
		V=\abs{E^{\boxtimes(n-1)}}q^{-2}\sum_{\mathclap{\substack{s,s'\in\F_q^*\\w,w'\in E\\sw\neq s'w'}}}\chi(t(s'-s))\sum_{z\in A}\chi(\varpi(z,sw-s'w')).
	\]
We use the \((A,B)\)-nondegeneracy of \(\varpi\) and orthogonality of \(\chi\) to obtain \(V=0\) so we are left with \(R^2\le U\).

We consider separately the summands in \(U\) where \(s\neq s'\) and where \(s=s'\). Define
    \[
        C\coloneqq\abs{E^{\boxtimes(n-1)}}q^{r-2}\sum_{\mathclap{\substack{s,s'\in\F_q^*\\w,w'\in\F_q^d\\sw=s'w'\\s\neq s'}}}\chi(t(s'-s))E(w)E(w')
    \]
and
    \[
        D\coloneqq\abs{E^{\boxtimes(n-1)}}q^{r-2}\sum_{\mathclap{\substack{s,s'\in\F_q^*\\w,w'\in\F_q^d\\sw=s'w'\\s=s'}}}\chi(t(s'-s))E(w)E(w')
    \]
so that \(U=C+D\).

For \(C\) take \(a\coloneqq\frac{s}{s'}\) and \(b\coloneqq s'\) to obtain
	\begin{dmath*}
		C = \abs{E^{\boxtimes(n-1)}}q^{r-2}\sum_{\mathclap{\substack{a,b\in\F_q^*\\w,w'\in\F_q^d\\a\neq1\\aw=w'}}}\chi(tb(1-a))E(w)E(w')
		= \abs{E^{\boxtimes(n-1)}}q^{r-2}\sum_{\mathclap{\substack{(a,w)\in\F_q^*\times\F_q^d\\a\neq 1}}}E(w)E(aw)\sum_{b\in\F_q^*}\chi(tb(1-a)).
	\end{dmath*}
From now on we assume that \(t\neq0\). By orthogonality we have that \(\chi(tb(1-a))=-1\) for each \(a\in\F_q^*\) with \(a\neq1\) and since \(E\) has projective index \(\alpha\) we find that
    \begin{dmath*}
        C = -\abs{E^{\boxtimes(n-1)}}q^{r-2}\sum_{\mathclap{\substack{(a,w)\in\F_q^*\times\F_q^d\\a\neq 1}}}E(w)E(aw)
        \le -\abs{E^{\boxtimes(n-1)}}\abs{E}q^{r-1}\alpha\left(1-\frac{2}{q}\right).
    \end{dmath*}

For \(D\) we find that
	\[
		D=\abs{E^{\boxtimes(n-1)}}q^{r-2}\sum_{\mathclap{(s,w)\in\F_q^*\times\F_q^d}}E(w)\le\abs{E^{\boxtimes(n-1)}}q^{r-2}\abs{E}q=\abs{E^{\boxtimes(n-1)}}\abs{E}q^{r-1}.
	\]
It follows that
    \[
        R^2\le U=C+D\le\abs{E^{\boxtimes(n-1)}}\abs{E}q^{r-1}\left(1-\alpha\left(1-\frac{2}{q}\right)\right)
    \]
so
    \[
        \abs{R}\le\abs{E^{\boxtimes(n-1)}}^{\frac{1}{2}}\abs{E}^{\frac{1}{2}}q^{\frac{r-1}{2}}\left(1-\alpha\left(1-\frac{2}{q}\right)\right)^{\frac{1}{2}}.
    \]
By assumption
    \[
        \abs{E}>q^{\frac{r+n-1}{n}}\left(1-\alpha\left(1-\frac{2}{q}\right)\right)^{\frac{1}{n}}
    \]
and it is the case that \(\abs{E^{\boxtimes(n-1)}}\ge\frac{\abs{E}^{n-1}}{q^{n-2}}\) so we see that
    \[
        \abs{R}\le\abs{E^{\boxtimes(n-1)}}^{\frac{1}{2}}\abs{E}^{\frac{1}{2}}q^{\frac{r-1}{2}}\left(1-\alpha\left(1-\frac{2}{q}\right)\right)^{\frac{1}{2}}<q^{-1}\abs{E^{\boxtimes(n-1)}}\abs{E}.
    \]
Since
    \[
        \nu(t)=q^{-1}\abs{E^{\boxtimes(n-1)}}\abs{E}+R
    \]
it must be that \(\nu(t)>0\). Since \(\nu(t)\) is the size of the set \(L_t\) we have that \(L_t\) is nonempty for each \(t\neq0\). That is, \(\F_q^*\subset\varpi(E^n)\).

A sharp example is located in \autoref{sec:sharp_example}.
\end{proof}

\section{Examples of forms}
We provide an alternative description of \((A,B)\)-nondegeneracy in terms of a generalized row space for forms. This will help us manufacture examples of \((A,B)\)-non-degenerate forms. We fix a field \(\F\) and denote by \(e_1,\dots,e_d\) the standard basis vectors for \(\F^d\).

\begin{defn}[\((k,A)\)-row space]
Given a form \(\varpi\in\form(\F^d,n)\), some \(k\in\N\), and a subspace \(A\le(\F^d)^{\otimes(n-1)}\) we define the \emph{\((k,A)\)-row space} of \(\varpi\) to be
    \[
        \row_{k,A}(\varpi)\coloneqq\set[(\eval_k(\varpi,e_1)(x),\dots,\eval_k(\varpi,e_d)(x))]{x\in A}.
    \]
\end{defn}

In the situation that \(A=(\F^d)^{\otimes(n-1)}\) we write \(\row_k(\varpi)\) rather than \(\row_{k,A}(\varpi)\) and refer to this space simply as the \emph{\(k\)-row space} of \(\varpi\). This generalizes the usual row space of a matrix in the following sense. Fix some \(1\le k\le n\) and consider \(\varpi\in\form(\F^d,n)\). Write \(\varpi\) in coordinates as
    \begin{dmath*}
        \varpi(x_1,\dots,x_n) = \sum_{\mathclap{i_1,\dots,i_n\in[d]}}u_{i_1,\dots,i_n}x_{1,i_1}\cdots x_{n,i_n}
        = \sum_{\mathclap{i_1,\dots,i_{k-1},i_{k+1},\dots,i_n\in[d]}}x_{1,i_1}\cdots x_{k-1,i_{k-1}}x_{k+1,i_{k+1}}\cdots x_{n,i_n}(u_{i_1,\dots,i_{k-1},i_{k+1},\dots,i_n}\cdot x_k).
    \end{dmath*}
We refer to the vectors \(u_{i_1,\dots,i_{k-1},i_{k+1},\dots,i_n}\in\F^d\) as the \emph{\(k\)-rows} of \(\varpi\), for the span of these vectors is the \(k\)-row space of \(\varpi\).


\begin{prop}
Given a form \(\varpi\in\form(\F^d,n)\) the following are equivalent to the \((A,B)\)-nondegeneracy of \(\varpi\) in the \(k^{\text{th}}\) coordinate:
    \begin{enumerate}
        \item \(\ker(\eval_{k,A,B}^\varpi)=0\)
        \item For each \(y\in B^*\) we have that \(\eval_{k,A,B}^\varpi(y)\neq0\).
        \item Given \(y\in B^*\) there is some \(m\in\N\) for which there are \(x_{i,j}\in(\F^d)^*\) for \(i\in[m]\) and \(j\neq k\) such that
            \[
                \sum_{i=1}^m\varpi(x_{i,1},\dots,x_{i,k-1},y,x_{i,k+1},\dots,x_{i,n})\neq0
            \]
        and \(\sum_{i=1}^m x_{i,1}\otimes\cdots\otimes x_{i,n}\in A^*\).
        \item \(B\cap(\row_{k,A}(\varpi))^\perp=0\)
    \end{enumerate}
\end{prop}

\begin{proof}
Recall that (1) was our original definition of \((A,B)\)-nondegeneracy in coordinate \(k\). We have that (2) is equivalent to (1) and also that (3) is equivalent to (2) by unpacking definitions. It remains to show that (2) is equivalent to (4).

Suppose that (4) holds, which means that given any \(y\in B^*\) we have that \(y\notin(\row_{k,A}(\varpi))^\perp\) and hence there exists some \(u\in\row_{k,A}(\varpi)\) such that \(u\cdot y\neq0\). Suppose that \(y=(y_1,\dots,y_d)\) and
    \[
        u=(\eval_k(\varpi,e_1)(x),\dots,\eval_k(\varpi,e_d)(x))
    \]
where \(x\in A\). We find that
    \begin{dmath*}
        0 \neq u\cdot y = \sum_{i=1}^d\eval_k(\varpi,e_i)(x)y_i
        = \sum_{i=1}^d\eval_k(\varpi,y_ie_i)(x)
        = \eval_k(\varpi,y)(x)
        = \eval_{k,A,B}^\varpi(y)(x),
    \end{dmath*}
which shows that \(\eval_{k,A,B}^\varpi(y)\neq0\). Thus, (4) implies (2).

We show that (2) implies (4) by proving the contrapositive. Suppose that there is some \(y\in B^*\) such that \(y\in(\row_{k,A}(\varpi))^\perp\). This implies that for every \(x\in A\) we have that
    \[
        0=(\eval_k(\varpi,e_1)(x),\dots,\eval_k(\varpi,e_d)(x))\cdot y=\eval_{k,A,B}^\varpi(y)(x),
    \]
which means that \(\eval_{k,A,B}^\varpi(y)=0\) and hence that (2) fails.
\end{proof}

We state a couple of corollaries for special cases. The first says that a form is non-degenerate precisely when its row space spans and the second says that a form \(\varpi\) is \((A,\row_{k,A}(\varpi))\)-non-degenerate when the restriction of the ususal dot product to \(\row_{k,A}(\varpi)\) is itself a non-degenerate bilinear form.

\begin{cor}
A form \(\varpi\in\form(\F^d,n)\) is non-degenerate in coordinate \(k\) if and only if \(\row_k(\varpi)=\F^d\).
\end{cor}

\begin{cor}
Given a form \(\varpi\in\form(\F^d,n)\) and a subspace \(A\le(\F^d)^{\otimes(n-1)}\) we have that \(\varpi\) is \((A,\row_{k,A}(\varpi))\)-non-degenerate if and only if for each \(x\in(\row_{k,A}(\varpi))^*\) we have that there exists some \(y\in\row_{k,A}(\varpi)\) such that \(x\cdot y\neq0\).
\end{cor}

We can obtain forms from \(n\)-ary relations on a finite set.

\begin{defn}[Form of a relation]
Given a relation \(\theta\subset S^n\) on a finite set \(S\) and a field \(\F\) the \emph{\(\F\)-form of \(\theta\)}, which is \(\theta_\F\in\form(\F^S,n)\), is given by
    \[
        \theta_\F(x_1,\dots,x_n)\coloneqq\sum_{\mathclap{s_1,\dots,s_n\in S}}1_\theta(s_1,\dots,s_n)x_{1,s_1}\cdots x_{n,s_n}
    \]
where \(1_\theta\colon S^n\to\F\) is the indicator function of \(\theta\).
\end{defn}

Of course, given a bijection between \(S\) and \([d]\) we can always think of \(\theta_\F\) as living in \(\form(\F^d,n)\). For forms of the form \(\theta_\F\) we have an alternative way to check for \((A,B)\)-degeneracy in a component.

\begin{prop}
\label{prop:relation_form_surjection}
Given a relation \(\theta\subset[d]^n\) we have that if \(\theta_\F\) is \((A,B)\)-non-degenerate in the \(k^{\text{th}}\) coordinate and \(B\supset\langle e_k\rangle\) then \(\pi_k\colon\theta\to[d]\) is surjective, where \(\pi_k\) is the \(k^{\text{th}}\) coordinate projection map.
\end{prop}

\begin{proof}
Suppose that \(\pi_k\) is not surjective with \(\alpha\in[d]\setminus\im(\pi_k)\). Consider a \(k\)-row \(u=u_{i_1,\dots,i_{k-1},i_{k+1},\dots,i_n}\). Observe that for any pure tensor \(x_1\otimes\cdots\otimes x_{n-1}\) we have that \(\eval_k(\theta_\F,\alpha)(x_1,\dots,x_{n-1})=0\), from which it follows that \(\eval_k(\theta_\F,\alpha)=0\). This implies that \(\langle e_k\rangle\in(\row_{A,k}(\theta_\F))^\perp\). Thus, if \(B\supset\langle e_k\rangle\) then we have that \(B\cap(\row_{A,k}(\theta_\F))^\perp\neq0\), which by the previous proposition shows that \(\theta_\F\) is not \((A,B)\)-non-degenerate in coordinate \(k\), establishing the contrapositive of our claim.
\end{proof}

Note that the preceding result places no restriction on \(A\) whatsoever.

The converse to this proposition is false in general. Take \(A\le(\F^d)^{\otimes(n-1)}\), take \(B\le\F^d\) with \(B\cap\langle(1,\dots,1)\rangle^\perp\neq0\), and consider the relation \(\theta=[d]^n\). Although \(\pi_k\colon\theta\to[d]\) is indeed surjective we have that \(\row_{k,A}(\varpi)\le\langle(1,\dots,1)\rangle\) so \((\row_{k,A}(\varpi))^\perp\ge\langle(1,\dots,1)\rangle^\perp\). This implies that \(B\cap(\row_{k,A}(\varpi))^\perp\neq0\) so we find that \(\varpi\) is \((A,B)\)-degenerate. We do have a partial converse, however.

\begin{prop}
\label{prop:relation_form_bijection}
Given a relation \(\theta\subset[d]^n\) we have that if \(\pi_k\colon\theta\to[d]\) is a bijection and for each \(\alpha\in[d]\) the subspace \(A\le(\F^d)^{\otimes(n-1)}\) contains
    \[
        e_{i_1}\otimes\cdots\otimes e_{i_{k-1}}\otimes e_{i_{k+1}}\otimes\cdots\otimes e_{i_n}
    \]
when \((i_1,\dots,i_{k-1},\alpha,i_{k+1},\dots,i_n)=\pi_k^{-1}(\alpha)\) then \(\theta_\F\) is \((A,B)\)-non-degenerate in the \(k^{\text{th}}\) coordinate.
\end{prop}

\begin{proof}
It suffices to show that \(\row_{k,A}(\theta_\F)\) contains each standard basis element \(e_\alpha\) for \(\alpha\in[d]\). Let \((i_1,i_2,\dots,i_{k-1},i_{k+1},\dots,i_n)=\pi_k^{-1}(\alpha)\). It follows that \(\row_{k,A}(\theta_\F)\) contains
    \[
        (\eval_k(\theta_\F,e_1)(\pi_k^{-1}(\alpha)),\dots,\eval_k(\theta_\F,e_d)(\pi_k^{-1}(\alpha)))=e_\alpha
    \]
as desired.
\end{proof}

Note that the preceding result places no restriction on \(B\) whatsoever.

\section{Applicability of the main theorem}
\label{sec:applicability}
Although the statement of \autoref{thm:generalized_result} refers to a wide range of possible arities \(n\) and dimensions \(d\), we will see that this result is vacuous is a nontrivial way for many values of \(n\) and \(d\).

\begin{prop}
Consider the case of our \autoref{thm:generalized_result} when \(\dim(\spn(E))=\ell\) and \(r=\ell^{n-1}\). We have that \autoref{thm:generalized_result} is vacuously true unless
    \begin{enumerate}
        \item \(\ell=1\),
        \item \(n=2\), or
        \item \(n=3\) and \(\ell=2\).
    \end{enumerate}
\end{prop}

\begin{proof}
Suppose that \(B=\spn(E)\), \(A=B^{\otimes(n-1)}\), and \(\dim(B)=\ell\). We have that \(\abs{E}\le q^\ell\) since \(E\subset B\cong\F_q^\ell\). If the estimate in \autoref{thm:generalized_result} holds then then
    \[
        q^{\frac{\ell^{n-1}+n-1}{n}}\left(1-\alpha\left(1-\frac{2}{q}\right)\right)^{\frac{1}{n}}<\abs{E}\le q^\ell.
    \]
This implies that
\[
    \ell^{n-1}-n\ell+n<2-\log_q(q-\alpha(q-2)).
\]
It is only the case that \(\ell^{n-1}-n\ell+n<2\) when
    \begin{enumerate}
        \item \(\ell=1\),
        \item \(n=2\), or
        \item \(n=3\) and \(\ell=2\).
    \end{enumerate}
Otherwise, the result is vacuously true since no set \(E\subset\F_q^d\) can satisfy the relevant bound.
\end{proof}

By restricting the form \(\varpi\) to \(B^{\otimes n}\) the preceding proposition effectively says that the only cases of interest are when \(d=1\) (which is relatively trivial), \(n=2\) (for which we give a result which is basically that of Hart and the second author\cite{hart2008} incorporating the projective index \(\alpha\)), and \(n=3\) and \(d=2\).

In this last case it is certainly possible for a set \(E\subset\F_q^2\) to fail to satisfy the bound in question, as the following example shows.

\begin{example}
Let \(A\) be the subgroup of \(\F_q^*\) of order \(\frac{q-1}{2}\) and take \(E\coloneqq A^2\). We have that \(\abs{E}=\frac{1}{4}(q-1)^2\) and that \(E\) has projective index \(\frac{q-3}{2(q-2)}\). We would have that \(\varpi(E^3)\supset\F_q^*\) for any non-degenerate form \(\varpi\) if
    \[
        \frac{1}{64}(q-1)^6=\abs{E}^3>q^6\left(1-\frac{q-3}{2(q-2)}\left(1-\frac{2}{q}\right)\right)=\frac{1}{2}q^6+\frac{3}{2}q^5.
    \]
\end{example}

When \(n=3\) and \(d=2\) our theorem applies only for
    \[
        \alpha>\frac{q^6-\abs{E}^3}{q^6-2q^5}
    \]
and hence only for sets \(E\subset\F_q^2\) with \(\abs{E}>\sqrt[3]{2}q^{\frac{5}{3}}\). That is, \autoref{thm:generalized_result} cannot say anything about sets \(E\) with \(\abs{E}\le\sqrt[3]{2}q^{\frac{5}{3}}\), independent of the projective index \(\alpha\).

\section{A sharp example}
\label{sec:sharp_example}
We show that \autoref{thm:generalized_result} gives a sharp bound by considering a non-degenerate form \(\varpi\in\form(q,d,n)\) and \(E\subset\F_q^d\) with \(E^{\boxtimes(n-1)}\subset A\), \(E\subset B\), and \(\abs{E}\approx q^{\frac{r+n-1}{n}}\) such that \(\varpi(E^{\boxtimes n})\not\supset\F_q^*\). That is, considering the projective index \(\alpha\) to be a constant independent of \(q\) we have that the factor \(1-\alpha\left(1-\frac{2}{q}\right)=O(1)\) so we can absorb it as a constant.

Let \(n=3\) and \(d=2\) and fix \(s\in\N\). Take \(q\) as large as possible so that \(s\mid(q-1)\) and not every element of \(\F_q^*\) can be written as a sum of two nonzero \(s^{\text{th}}\) powers. Let \(\Gamma\) be the multiplicative subgroup of \(\F_q^*\) of order \(\frac{q-1}{s}\) and note that since \(\Gamma\) consists of the nonzero \(s^{\text{th}}\) powers in \(\F_q\) we have that \(\Gamma+\Gamma\not\supset\F_q^*\) where
    \[
        \Gamma+\Gamma\coloneqq\set[\gamma_1+\gamma_2]{\gamma_1,\gamma_2\in\Gamma}.
    \]

Define \(E\coloneqq\Gamma^2\). We have that \(r=\dim((\F_q^2)^{\otimes(3-1)})=4\) so
    \[
        \abs{E}=\left(\frac{q-1}{s}\right)^2\approx q^2=q^{\frac{4+3-1}{3}}=q^{\frac{r+n-1}{n}}
    \]

Take \(\theta\coloneqq\set{(1,1,1),(2,2,2)}\subset[2]^3\) and define \(\varpi\coloneqq\theta_{\F_q}\). We have that \(\varpi\) is non-degenerate in coordinate \(3\), yet
    \[
        \varpi(E^{\boxtimes3})=\Gamma+\Gamma\not\supset\F_q^*,
    \]
demonstrating that \autoref{thm:generalized_result} is sharp. This is really little more than the assertion that up to the highest-order term the bound in our main theorem for the case \(n=3\) and \(d=2\) cannot be improved beyond the trivial bound of ``approximately all of \(\F_q^2\)''. However, in the proceding section we give examples of nontrivial application of our theorem when taking the term \(1-\alpha\left(1-\frac{2}{q}\right)\) into account.

\section{Examples of nontrivial bounds}
\label{sec:nontrivial_examples}
We consider the following family of sets in this section.

\begin{defn}[Omphalos]
We say that a set \(E\subset\F_q^2\) is a \emph{\((q,k,\ell)\)-omphalos} when
    \[
        E=\bigcup_{h\in H}E_h
    \]
where \(H\) is a set of \(k\) distinct lines through the origin in \(\F_q^2\) and each \(E_h\) consists of exactly \(\ell\) nonzero points from \(h\).
\end{defn}

Our bound for the \(n=3\) and \(d=2\) case may be made more succinct when \(E\) is a \((q,k,\ell)\)-omphalos.

\begin{prop}
\label{prop:omphalos_bound}
Suppose that \(E\) is a \((q,k,\ell)\)-omphalos and that \(\varpi\) is a non-degenerate ternary form on \(\F_q^2\). If
    \[
        k^3\ell^3>q^6-(\ell-1)q^5
    \]
then \(\varpi(E^2)\supset\F_q^*\).
\end{prop}

\begin{proof}
Apply \autoref{thm:generalized_result} to \(E\) in the case that
    \[
        \alpha\coloneqq\frac{\ell-1}{q-2}.
    \]
Since \(\abs{E}=k\ell\) cubing the bound
    \[
        \abs{E}>q^2\left(1-\alpha\left(1-\frac{2}{q}\right)\right)^{\frac{1}{3}}
    \]
yields the desired inequality.
\end{proof}

Omphaloi may be constructed from cosets of multiplicative subgroups of \(\F_q^*\).

\begin{example}
\label{ex:coset}
Let \(\Gamma\) be a subgroup of \(\F_q^*\) of order \(\frac{q-1}{s}\) and take \(H\subset\F_q^*\) where \(\abs{H}=r\) and if \(h_1,h_2\in H\) with \(h_1\neq h_2\) then \(h_1\Gamma\neq h_2\Gamma\). That is, let \(H\) consist of representatives of \(r\) distinct cosets of \(\Gamma\). Define
    \[
        E\coloneqq\set[x(1,y)\in\F_q^2]{x,y\in H\Gamma}.
    \]
Note that \(\Gamma\) is a \((q,k,l)\)-omphalos where
    \[
        k=l=\frac{r(q-1)}{s}.
    \]
\end{example}

As a corollary to \autoref{prop:omphalos_bound} we have the special case where \(E\) is as in the above example.

\begin{cor}
\label{cor:coset_omphalos}
Taking \(E\) from \autoref{ex:coset} we have that if
    \[
        (q-1)^6r^6+(q-1)s^5q^5r-s^6(q^6+q^5)>0
    \]
and \(\varpi\) is a non-degenerate ternary form on \(\F_q^2\) then \(\F_q^*\subset\varpi(E^3)\).
\end{cor}

We illustrate the application of this corollary in our last example. Further discussion of this particular case appears in the subsequent section.

\begin{example}
Consider the set \(E\subset\F_q^2\) from \autoref{ex:coset} where \(q=160001\), \(s=20\), and \(r=16\) and take \(\varpi\) to be the ternary dot product from \autoref{example:ternary_dot_product}. Since
    \[
        (q-1)^6r^6+(q-1)s^5q^5r-s^6(q^6+q^5)>0
    \]
in this case and \(\varpi\) is non-degenerate we have that every nonzero member of \(\F_q^*\) may be written as
    \[
        \varpi(h_1\gamma_1(1,h_2\gamma_2),h_3\gamma_3(1,h_4\gamma_4),h_5\gamma_5(1,h_6\gamma_6))
    \]
where the \(h_i\) are from a fixed set \(H\) consisting of \(r=16\) coset representatives of the subgroup \(\Gamma\) of \(\F_q^*\) of order \(\frac{q-1}{s}=8000\) and the \(\gamma_i\) are members of \(\Gamma\). Thus, each member of \(\F_q^*\) is of the form
    \[
        h_1h_3h_5\psi_1(1+h_2h_4h_6\psi_2)
    \]
where \(\psi_1\) and \(\psi_2\) are \(20^{\text{th}}\) powers in \(\F_q\) and the \(h_i\) belong to \(H\). This also implies that each member of \(\F_q*\) can be written as
    \[
        h_1\gamma_1h_3\gamma_3h_5\gamma_5+h_1\gamma_1h_2\gamma_2h_3\gamma_3h_4\gamma_4h_5\gamma_5h_6\gamma_6
    \]
and hence that
    \[
        A\cdot A\cdot A+A\cdot A\cdot A\cdot A\cdot A\cdot A\supset\F_q^*
    \]
when \(A=H\Gamma\).
\end{example}

\section{Open questions and future prospects}
In this section we describe some open problems that arose from this paper and some possible extensions of these results. 

\begin{itemize}
    \item In \autoref{sec:applicability} we note that the only nontrivial cases in which our result can apply are the cases where \(n=2\) or when both \(n=3\) and \(d=2\). What further assumptions on \(E\) or \(\varpi\) might yield nontrivial bounds for other values of \(n\) and \(d\)? Are there useful applications of our stronger result in the \(n=2\) along the lines of what was done in \cite{hart2008}?
    \item A natural multi-linear form that falls within the scope of our results is \(\det(x^1,x^2, \dots, x^d)\) where \(x^j\in\F_q^d\). This form was studied systematically in \cite{CHIK2010} under a variety of structural assumptions on the underlying set \(E\). Using the main result of this paper, we obtain similar exponents under different structural assumptions having to do with tensor products. It would be very interesting to reconcile and unify these results, particularly in the context of a regime that takes us beyond the \(n=3\) and \(d=2\) case.
    \item Although \autoref{sec:sharp_example} shows that in general \autoref{thm:generalized_result} gives us a bound on the order of \(q^2\) in the case that \(n=3\) and \(d=2\), it is natural to wonder whether the bound given in \autoref{cor:coset_omphalos} is the best possible. That is, given a prime \(q\) and \(s\mid q-1\) let \(f(q,s)\) be the least \(r\in\N\) such that
        \[
            (q-1)^6r^6+(q-1)s^5q^5r-s^6(q^6+q^5)>0
        \]
    and let \(g(q,s)\) be the least \(r\) such that for any non-degenerate ternary form \(\varpi\) on \(\F_q^2\) we have that \(\F_q^*\subset\varpi(E^3)\) where \(E\) is as in \autoref{ex:coset} for any possible choice of coset representative \(H\) of size \(r\). Necessarily we have that \(f(q,s)\ge g(q,s)\), but how much smaller can \(g(q,s)\) be in general?
\end{itemize}

\printbibliography

\end{document}